\documentclass[letterpaper, 10 pt, conference]{ieeeconf}  

\IEEEoverridecommandlockouts                              

\RequirePackage{graphicx}

\usepackage[utf8]{inputenc}
\usepackage[T1]{fontenc} 
\usepackage{hyperref}   
\usepackage{url}         
\usepackage{booktabs}  
\usepackage{amsfonts}      
\usepackage{nicefrac}       
\usepackage{microtype}      
\usepackage{xcolor}     

\usepackage{amsmath}
\usepackage{amsfonts}
\usepackage{amssymb}
\usepackage{cuted}
\usepackage{url}
\usepackage{comment}

\usepackage[linesnumbered,ruled]{algorithm2e}

\newcommand{\ex}{{\bf\sf E}}               

\newcommand{\R}{\mathbb{R}}

\newcommand\Complex{\mathbb C}

\newcommand{\red}[1]{\textcolor{red}{#1}}

\newcommand{\bX}{{\bf X}}  

\newcommand{\call}{{\mathcal L}}
\newcommand{\calh}{{\mathcal H}}

\newcommand{\calk}{{\mathcal K}}

\newcommand{\calf}{{\mathcal F}}

\newcommand{\ra}{\rightarrow}           

\def\R{\mathbb{R}}
\def\ex{\mathbb{E}}

\newcommand{\Real}{\mathbb R}

\newcommand{\bbT}{\mathbb T}

\newcommand{\g}{\lambda}                
   
\newcommand{\ta}{\theta}                

\newtheorem{defn}{Definition}

\newtheorem{thm}{Theorem}[section]

\newtheorem{rem}{Remark}[section]
\newtheorem{asm}{Assumption}

\newtheorem{exm}{Example}

\title{\LARGE \bf Mean Field Control by Stochastic Koopman Operator \\ via a Spectral Method}

\author{Yuhan Zhao$^{1}$, Juntao Chen$^{2}$, Yingdong Lu$^{3}$ and Quanyan Zhu$^{1}$
\thanks{$^{1}$ Y. Zhao and Q. Zhu are with New York University. Email: \{yhzhao, qz494\}@nyu.edu}%
\thanks{$^{2}$ J. Chen is with Fordham University. E-mail: jchen504@fordham.edu}%
\thanks{$^{3}$ Y. Lu is with IBM Research. E-mail: yingdong@us.ibm.com}%
}

\begin{document}

\maketitle
\thispagestyle{empty}
\pagestyle{empty}

\begin{abstract}

Mean field control provides a robust framework for coordinating large-scale populations with complex interactions and has wide applications across diverse fields. However, the inherent nonlinearity and the presence of unknown system dynamics pose significant challenges for developing effective analytic or numerical solutions. There is a pressing need for data-driven methodologies to construct accurate models and facilitate efficient planning and control.
To this end, we leverage Koopman operator theory to advance solution methods for mean field control problems. Our approach involves exploring stochastic Koopman operators using spectral analysis techniques. Through Koopman decomposition, we derive a linear model for mean field control problems in a data-driven fashion. Finally, we develop a model predictive control framework to achieve robust control and reduce the computational complexity for mean field control problems, thereby enhancing the efficacy and applicability of mean field control solutions in various domains.

\end{abstract}


\section{Introduction}
\label{sec:intro}

Mean field approximation and mean field control have been widely used in various areas to characterize interactions between numerous agents, such as robotics \cite{elamvazhuthi2019mean,fernando2021online}, transportation \cite{huang2021dynamic}, and biology \cite{bick2020understanding}. A typical example is flow control over traffic networks, where a traffic planner aims to coordinate a large number of drivers and reduce the overall traveling time. The traveling time is proportional to the traffic density on the road. The more drivers who select the same path, the heavier the congestion is, and the more time is required for traveling. Therefore, the traffic planner needs effective control strategies to coordinate drivers that account for collective driver interactions. The optimal coordination strategy can be found by a mean-field control problem.
Similar control strategy designs are also seen in robot swarm coverage \cite{elamvazhuthi2018nonlinear} and herding \cite{elamvazhuthi2020controllability}.

Motivated by these applications, we consider the following control problem (similar to those in~\cite{ELLIOTT20133222}). For each $k=0,1, \ldots, N-1$, we have
\begin{align}
\label{eqn:MFC_dynamics}
x_{k+1} =b (x_k, \ex[x_k], u_k, \ex[u_k]) + \sigma (x_k, \ex[x_k], u_k, \ex[u_k]) w_k
\end{align}
with Lipschitz continuous functions $b: \Real^n\times \Real^n\times \Real^m \times \Real^m \ra \Real^n$ and $\sigma: \Real^n\times \Real^n\times \Real^m \times \Real^m \ra \Real^{n\times n}$, and $w_k$ a martingale difference sequence in $\Real^n$ (see definition and basic properties of martingale and martingale differences in~\cite{karlin2012first}). $x_k\in \Real^n$, $u_k\in \Real^m$, and $w_k\in \Real^n$ represent the state, control, and noise of the system. The cost function of the system is defined as
\begin{align}
J(u, x^0)=&\ex\left[\sum_{k=0}^{N-1} [F (x_k, \ex[x_k])+ C(u_k, \ex[u_k]) ]\right] \nonumber\\& + \ex \left[ G(x_N, \ex[x_N]) \right] + \ex \left[ H( u_N, \ex[u_N]) \right], \label{eqn:MFC_costs}
\end{align}
with $F,G,C,H$ denote the elements of the cost function associated with states (including the average states) and controls (including the average controls), as well as their terminal values at time $N$. 
Our goal is to solve the following optimal control problem with initial state $x^0\in \Real^n$,
\begin{align}
\label{eqn:MFC}
J^*(x^0)=\inf_{u\in U} J(u, x^0),
\end{align}
with $U$ denotes all the stationary random control policies that only depend on the states. More specifically, given a state $x$, for $u\in U$, $u(x_k)$ is a random variable taking values in $\Real^m$ with $\pi(x_k)$ denoting its expectation. \eqref{eqn:MFC} belongs to the family of mean-field control problems, for detailed models and analysis, see, e.g.~\cite{carmona2018probabilistic}.

The key feature of a mean-field control problem for systems with large populations is that the evolution of each individual depends not only on the state of the individual but also on the probabilistic distribution of states for the entire population. This is intended to model the collective behavior resulting from interactions between agents. This feature is summarized by the system dynamics in the form of both state and control variables, as well as their probability laws or statistics as in~\eqref{eqn:MFC}. Such dynamics is a consequence of a limiting process of letting the total population goes to infinity. As an 
active research topic in analyzing systems that consist of a large number of homogeneous agents, mean-field control is highly related to the problem of the mean field game, a similar problem corresponding to a different limiting process. For a detailed analysis of the limiting processes and their connections, see e.g.~\cite{MFGvMFC}. While the mean field game, partly due to its further connections to subjects such as stochastic differential equations, draws considerable attention in the research communities, the studies on mean field control are relatively less thrived. 
A linear-quadratic formulation with a similar approach to the mean-field control has been studied in~\cite{ELLIOTT20133222}, where necessary and sufficient conditions for solvability are obtained. Moreover, the problem is reduced to a matrix dynamic optimization problem, and the optimal control is then subsequently obtained and verified.


In this paper, we propose a new approach for the family of mean field control problems described  in~\eqref{eqn:MFC_dynamics},~\eqref{eqn:MFC_costs}, and~\eqref{eqn:MFC} based on model predictive control and spectral theory for stochastic Koopman operators. 

From a theoretical perspective, the Koopman operators approach provides an effective and efficient scheme for quantifying and approximating the dynamics of the controlled process. It has been applied successfully to various control problems, see e.g.~\cite{korda2018linear,proctor2018generalizing}. Recent developed spectral theory for the Koopman operator, see e.g.~\cite{colbrook2024rigorous}, captures both the discrete and absolution continuous portion spectrum  of the Koopman operator, thus is able to provide convergence guarantees of recovery of the spectral densities and avoids spectral pollution with sufficient data. We extend the spectral analysis to stochastic Koopman operator with control, especially for mean field problems, and demonstrate that this convergence further lead to an asymptotic optimality result for the optimal control problem, similar to those obtained in~\cite{6144708,LU2023111023}. These results then allow us to formulate a model predictive control method for mean-field control problems aimed at optimal nonlinear control of systems with a large population of agents. 

The rest of the paper is organized as follows. In Sec.~\ref{sec:preliminaries}, we will provide necessary preliminaries and background materials; in Sec.~\ref{sec:analysis}, we will adapt the spectral analysis to the random Koopman operator to the mean-field control system including the asymptotic optimality of the approximated control problem in Sec.~\ref{sec:asm_opt}; in Sec.~\ref{sec:MPC}, we introduce our model predictive control formulation for MFC;  and the paper is concluded in Sec.~\ref{sec:conclusions} together with future research.

\section{Preliminaries}
\label{sec:preliminaries}

\subsection{Stochastic Koopman Operator for Mean Field Control Problem}

\subsubsection{Stochastic Koopman Operator without Control}

To use the Koopman operator theory, we first consider a dynamical system without control
\begin{align}
\label{eqn:MFC_dynamics_1}
\left\{\begin{array}{cc}x_{k+1} =&b (x_k, \mu_k) + \sigma (x_k, \mu_k) w_k,
\\ \mu_{k+1}=& b (x_k, \mu_k)\end{array}\right.
\end{align}
with $\ex[x_k]$ is denoted as $\mu_k$. We write the state vector as $y_k=(x_k, \mu_k)\in M \subseteq \Real^{2n}$. Then, \eqref{eqn:MFC_dynamics_1} represents a dynamical system in the form of $y_{k+1}= T (y_k)$, with $T$ being a random function\footnote{For simplicity, we use $M$ to denote the space of the aggregated state vector $y$ in the following. $M$ may have different dimensions when the underlying dynamical system changes. For example, $M \subseteq \R^{2(m+n)}$ for the dynamical system \eqref{eqn:MFC_dynamics_2}.}.

\begin{defn}
The stochastic Koopman operator, a bounded linear operator, $\calk$ is defined on the space of $\{f:M\ra \R :  \ex[f^2]<\infty\}$ as $\calk f(y) = \ex[f(T(y)]$, for any $y\in M$.
\end{defn}

\begin{rem}
Stochastic Koopman operators are studied in~\cite{vcrnjaric2020koopman} and~\cite{Sinha2020}, with dynamical system setup from standard textbook such as~\cite{katok1995introduction} and~\cite{palis2012geometric}.
\end{rem}

\subsubsection{Stochastic Koopman Operator with Control}

Koopman Operator approach has been generalized to controlled dynamical systems, see e.g.~\cite{korda2018linear,proctor2018generalizing}. Here, we adopt a similar approach for our stochastic Koopman operator. More specifically, we denote $\ex[u_k]$ as $\rho_k$ and let $v_k$ be a martingale difference sequence representing the randomization of the control.
Then we arrive at the following from \eqref{eqn:MFC_dynamics}:
\begin{align}
\label{eqn:MFC_dynamics_2}
\left\{\begin{array}{cc}
    x_{k+1} =&b (x_k, \mu_k, u_k, \rho_k) + \sigma (x_k, \mu_k, u_k, \rho_k) w_k,\\ 
    \mu_{k+1}=& b (x_k, \mu_k, u_k, \rho_k), \\ 
    u_{k+1}=& \pi(x_{k+1}) + v_{k+1},\\
    \rho_{k+1}=& \pi(x_{k+1}). \end{array}
\right.
\end{align}
The dynamical system in~\eqref{eqn:MFC_dynamics_2} defines the operator $T$, i.e. $(x_{k+1}, \mu_{k+1}, u_{k+1}, \rho_{k+1}) = T((x_{k}, \mu_{k}, u_{k}, \rho_k))$. Therefore, on the space $\{f:\Real^n \times \Real^n \times \Real^m\times \Real^m\ra \Real: \ex[f^2]<\infty\}$, the stochastic Koopman operator is defined as $\calk f(y) =\ex f(T(y))$, where $y=(x_{k+1}, \mu_{k+1}, u_{k+1}, \rho_{k+1}) \in M \subseteq \R^{2(m+n)}$.
\begin{rem}
An alternative and more abstract treatment for the Koopman operator for control that requires infinite sequence and shift operator is presented in~\cite{korda2018linear}. Our approach will be aligned with a more straightforward approach also discussed in both~\cite{korda2018linear} and~\cite{proctor2018generalizing}. Basically, the control is assumed to be stationary random policies. Hence, they will be generated randomly with the previous control as parameters. Thus, both the position and mean of the current process are present in consistent with mean-field control. 
\end{rem}

\subsection{Spectral Theorem and Spectral Measure}

As an operator between Hilbert spaces, the Koopman operator $\calk$ can be diagonized with the help of the classic results of the Spectral Theorem for the operators. 

\subsubsection{Spectral Theorem}
The spectral theorem is a generalization of diagonalization of matrices, and detailed discussion can be found in traditional functional analysis textbooks, such as~\cite{reed1981functional},~\cite{hille1996functional} and~\cite{yosida2012functional}, and a more recent and excellent book~\cite{simon2015operator}. In the case of matrices, a normal matrix (satisfying $AA^*=A^*A$) possesses the following decomposition, $A=\sum_{i=1}^d\g_i E_i$ with $E_i$ being the projection to the eigenspace corresponding to the eigenvalue $\g_i$, $i=1,2,\ldots d$. Similar results can be obtained for normal linear operators on Hilbert spaces. To see that, let us start with the concept of \emph{resolution of identity}.

\begin{defn}[Resolution of Identity]
A \emph{resolution of Identity} is a function $t\mapsto E_t$ from $\Real$ to the orthogonal projections on $H$ satisfying
\begin{itemize}
\item
monotonicity: $E_t\le E_s$ for any $t<s$,
\item
regularity: $\lim_{t\ra -\infty}E_t=0, \quad \lim_{t\ra \infty}E_t=Id$,
\item
continuity: $\lim_{t\downarrow s} E_t=E_s$.
\end{itemize}
\end{defn}
Then, the following version of the spectral theorem applies to the self-adjoint operator ($T=T^*$), which is, of course, normal.  
\begin{thm}[Spectral Theorem~\cite{simon2015operator}]
For any bounded self-adjoint operator $A$ on a Hilbert space $\calh$, there exists a resolution of identity, $E_t$, supported on $[-\|A\|, \|A\|]$ with $\|A\|$ being the operator norm of $A$. Furthermore, the operator $A$ has a spectral integration representation  $A=\int t dE_t$. 
\end{thm}
\begin{rem}
In addition to diagonalization, spectral theorem also provides representations to derived operators based on $A$, for example, $A^n = \int t^n dE_t$ and $\exp[A]= \int e^t dE_t$. Some of these representations are instrumental in subsequent development.
\end{rem}

\subsubsection{Spectral Measure}

In one sentence, the spectral measure is a projection-valued measure, denoted by $\xi$, on the spectrum of the operator that corresponds to the integration in the Spectral Theorem. Even though the above version of the Spectral Theorem is stated for the self-adjoint operator whose spectrum is real, all the results can be extended to the normal operator on a complex Hilbert space whose spectrum is a subset in the complex plane. 

\subsection{Lebesgue Decomposition of Measures}

Now, for a real-valued measure $\nu$ on $M$, Lebesgue decomposition, see, e.g.~\cite{halmos2013measure,rudin1987real,bogachev2007measure}, produces the following decomposition
\begin{align}
d\nu(y) =& \underbrace{\delta(y-\ta) dy}_{\text{discrete component}} + \underbrace{\psi(y)dy}_{\text{absolutely continuous component}}\nonumber \\ &+ \underbrace{d\nu^{SC}(y)}_{\text{singular continuous component}}.
\label{eqn:Leb_dec}
\end{align}
The discrete component ($\delta$ function, point measures) corresponds to the isolated eigenvalues. The absolute continuous component represents the part of the measures that has absolutely continuous density function with respect to the Lebesgue measure.
The Singular continuous component represents the part of the measures that is continuous with respect to the Lebesgue measure but not absolutely continuous, most likely driven by some Cantor functions.



\section{Spectral Analysis of the stochastic Koopman Operator}
\label{sec:analysis}

In this section, we will provide details on the spectral analysis of our stochastic Koopman operator for controlled system. The analysis leads to Algorithms~\ref{alg:cap} and~\ref{alg:2}, as well as their convergence in   Theorem~\ref{thm:DS_conv}. As a consequence, Theorem~\ref{thm:convergence_of_the_DS} demonstrates that this can further lead to the convergence of the controlled systems. 

\subsection{Measure-Preserving Dynamical Systems}

For ease of exposition, we will restrict our attention to measure-preserving dynamical systems. For more general dynamical systems, our approach can be generated through Galerkin approximation methods, see, e.g.~\cite{colbrook2024rigorous}. In a measure-preserving dynamical system defined as iterations of map $T$ as in\eqref{eqn:MFC_dynamics}, the pushforward by the map $T$ does not change the measure $\mu$\footnote{With a little abuse of notation, here we only use $\mu$ to represent a measure.}, i.e. for any Borel measurable subset $A\subseteq \bX$, $\mu(T^{-1}(A)) =\mu(A)$. Equivalently, this means that the Koopman operator $\calk$ is an isometry, i.e. $\|\calk g\|=\|g\|$ for and $g\in D(\calk)$. This, in turn, means that the spectrum, i.e. $\sigma(\calk):=\{ \g \in \Complex: \g I - \calk \hbox{ is not invertible}\}$, is contained in the unit circle $\bbT$ on the complex plane. In the case of a random dynamical system, we assume that the measure-preserving property is satisfied almost surely. Therefore, the spectral measure is a projection-valued measure on the Borel algebra on the unit circle. In summary, we have, 
\begin{asm}
\label{asm:measure_preserving}
$T$ is almost surely measure-preserving.
\end{asm}

We note that the transformation $[-\pi, \pi]/\{-\pi, \pi\}\ra \bbT, \ta \mapsto e^{i\ta}$ establishes an isometry between the unit circle and $[-\pi, \pi]/\{-\pi, \pi\}$, which is also denoted as $[-\pi, \pi]_{per}$. For any functions $f,g \in L^2(M)$ and Borel set $U \subseteq[-\pi, \pi]_{per}$, we define $\nu_{f,g}(U):=\langle f, \xi(U) g\rangle$ using the inner product of $L^2(M)$. By definition, $\xi(U)$ is a projection operator. Hence, $\xi(U)g$ is in the Hilbert space $L^2(M)$. Thus, $\nu_{f,g}(\theta)$ is a real-valued measure on $[-\pi, \pi]_{per}$. 
Since there can be atoms in $[-\pi, \pi]_{per}$, the integration using $\nu_{f,g}$ on $[-\pi, \pi]_{per}$ is generally written as $\int_{[-\pi, \pi]_{per}} d\nu_{f,g}(\theta)$.
Note that a special case is to set $f=g$ (e.g., see~\cite{colbrook2024rigorous}), where $\nu_{g,g}(U) = \langle g, \xi(U) g \rangle$ becomes a scalar measure and represents the spectral measure of $\calk$ with respect to the function $g$. 

\begin{asm}
\label{asm:no_singular_continuity}
The Lebesgue decomposition of $\nu_{f,g}$ does not contain a singular continuous component.
\end{asm}

Under the Assumption \ref{asm:no_singular_continuity} (which is the case for most of the measures encountered in practical applications), we apply the Lebsegue decomposition~\eqref{eqn:Leb_dec} and obtain 
\begin{align}
\label{eqn:Leb_dec_nu}
    d\nu_{f,g}(y) = \sum_{\g=e^{i\theta} \in \sigma_p(\calk)} \langle P_\g g, f\rangle \delta(y-\ta) dy +\psi_{f,g}(y) dy.
\end{align}
The Lebesgue decomposition \eqref{eqn:Leb_dec_nu} tells the discrete component of the spectrum $\sigma_p(\calk)$, which contains the eigenvalues $\lambda_1, \lambda_2, \dots, \in \sigma_p(\calk)$ and the corresponding eigenfunctions $\phi_\lambda$, and the absolutely continuous component, which is the density $\psi_{f,g}(y)$.
Besides, Assumption~\ref{asm:no_singular_continuity} also implies that the integral on $[-\pi, \pi]_{per}$ can be interpreted as Riemann integral.
Therefore, any function $g\in L^2(M)$ can be written as 
\begin{align*}
g(y) = \sum_{\g\in \sigma_p(\calk)} c^g_\g\phi_\g(y) + \int_{-\pi}^\pi \phi_{\ta, g}(y) d\ta.
\end{align*}
Here, 
$c^g_\lambda$ are expansion coefficients; $\phi_{\ta, g}$ is absolutely continuous components of the function-valued measure $d\xi(\ta) g$ and 
$\psi_{f,g} = \langle \phi_{\theta,g}, f \rangle$. 
This leads to the following representations of the system dynamics with $y$ being the state variable:
\begin{align}
\nonumber
\ex[g(y_n)] =&\calk^ng(y_0)
\\= & \sum c_\g \g^n \phi_\g(y_0) + \int_{-\pi}^{\pi} e^{in\ta} \phi_{\ta, g}(y_0)  d\ta,\label{eqn:spectrum_exp}
\end{align}
which plays a crucial role in prediction and control.
Therefore, it is crucial that we can recover both $\phi_\g$ and $\phi_{\ta, g}$ from data. 

\subsection{Data-Driven Spectral Measure Characterization of the Koopman Operator}

\subsubsection{Characterizing A Real-valued Measure on An Interval} 

$\nu_{f,g}(\ta)$, as a real-valued measure on $[-\pi, \pi]_{per}$, can be completely described by its characteristic function obtained by the Fourier transform
\begin{align*}
\calf \nu_{f,g} (\omega) = \frac{1}{2\pi} \int_{-\pi}^\pi e^{-i \omega\ta} d\nu_{f,g}(\ta).
\end{align*}
Fourier series, the discrete counterpart, see e.g.~\cite{edwards1979fourier}, connects to the dynamical system more directly,  
\begin{align}
{\hat \nu}_{f,g}(n) :=&\frac{1}{2\pi} \int_{-\pi}^\pi e^{-in\ta}d\nu_{f,g}(\ta)=\frac{1}{2\pi}\int_\bbT \g^n d\nu_{f,g}(\g)\nonumber \\=&\frac{1}{2\pi}\langle f, \calk^n g\rangle. \label{eqn:Fourier_coefficient}
\end{align}
The last quantity is the inner product of the observation function $f$ and the $n$-th observation of the dynamical system using the observation function $g$. The implication of~\eqref{eqn:Fourier_coefficient} is that an estimation of $\langle f, \calk^n g\rangle$, which is straightforward from data, produces Fourier coefficients for the desired measure.

\subsubsection{Recovering the Spectral Measure} 

The problem of recovering the spectral measure of the Koopman operator amounts to obtaining all or some of ${\hat \nu}_{f,g}(n)$ through calculations or approximations. The data-driven approach aims at exactly that through the sampling method. More specifically, for any desired accuracy, there exists $N>0$, such that the following expression produces a desirable  approximation  of ${\hat \nu}_{f,g}(n)$,
\begin{align*}
\nu_{f,g,N}(\ta) =&\sum_{n=-N}^N \phi\left(\frac{n}{N}\right) {\hat \nu}_{f,g}(n) e^{in\ta} \\=& \frac{1}{2\pi}\sum_{n=-N}^{-1} \phi\left(\frac{n}{N}\right) \overline{\langle f, \calk^{-n} g\rangle}e^{in\ta}\\ &+ \frac{1}{2\pi}\sum_{n=0}^N \phi\left(\frac{n}{N}\right) \langle g, \calk^{-n} f\rangle e^{in\ta},
\end{align*}
with $\phi(\cdot)$ being the standard filter function. In a nutshell, the $\nu_{f,g,N}(\ta)$ provides a recovering of the value $\ex[f(y_k)]$ using $N$ trajectory points observed by $f$ and $g$.
We summarize the spectral measure recovery algorithm in Alg.~\ref{alg:cap}.

\begin{algorithm}
    \SetKwInOut{Input}{Input}
    \SetKwInOut{Output}{Output}

    \Input{Filter function $\phi$ and observation functions $f,g\in L^2(M)$}
    Sample consecutive observation data $g(y_n) := \calk^n g(y_0)$, $n = 0, \dots, N$ \;
    Compute Fourier coefficients $a_n= \frac{1}{2\pi}\langle f, \calk^n g\rangle$ for $0\le n\le N$\;
    set $a_{-n}={\bar a_n}$\;
    \Output{$\nu_{f,g,N}(\ta)=\sum_{n=-N}^N \phi(n/N)a_ne^{in\ta}$}
    
    \caption{Recovering spectral measure $\nu_{f,g}(\ta)$ }\label{alg:cap}
\end{algorithm}

\subsubsection{Recovering the Eigenvalues and Eigenfunctions} 

A key consequence of recovering the spectral measure $\nu_{f,g}$ through data is the recovering of eigenvalues $\g\in \sigma_p(\calk)$ and their associated eigenfunctions. More specifically, Theorem 3.3 in~\cite{colbrook2024rigorous} demonstrates that under mild assumptions that all the operators easily satisfy, the eigenvalues of $\calk$, as the discrete portion of the spectral measure, can be recovered fully through the Fourier series method. Without causing confusion, we denote the recovered eigenvalues by $\sigma_p(\calk)$.
In practice, a truncated Fourier series will be used to approximate these eigenvalues. We can also set thresholds (e.g., take the first largest $N$ eigenvalues) to ensure $|\sigma_p(\calk)|$ is finite. For each eigenvalue $\g$, it is demonstrated in~\cite{Mezic2005} that the associated eigenfunction $\phi_\g$ takes the form of
\begin{align} \label{eqn:eigenfn}
\phi_\g (y) =\lim_{n\ra\infty} \frac{1}{n}\sum_{j=0}^{n-1}e^{i2\pi j\g} \ex[h(T^j(y))], \quad \forall y\in M,
\end{align}
where $h$ is some observation function in the Hilbert space. It can be again approximated through truncation.
A practical approach to estimate the coefficients $c^g_\lambda$ is to disregard the continuous part of the spectral measure and use the least square approximation. This is because the continuous part is generally small for common dynamical systems (except for chaotic systems).
More specifically, given an observation function $g$, we sample observation trajectory $\{ g(y_i)\}_{i=1}^N, \{ \phi_\lambda(y_i) \}_{i=1}^N$ for $\lambda \in \sigma_p(\calk)$. Then, we solve the following least square approximation problem to obtain $c^g_\lambda$:
\begin{equation}
\label{eqn:leastsquare}
    \min_{ \{c^g_\lambda \}_{\lambda \in \sigma_p(\calk)} } \quad \sum_{i=1}^N \big\Vert g(y_i) - \sum_{\lambda \in \sigma_p(\calk)} c^g_\lambda \phi_\lambda(y_i) \big\Vert^2_2.
\end{equation}
We summarize the procedure in Alg.~\ref{alg:2}.

\begin{algorithm}
\SetKwInOut{Input}{Input}
\SetKwInOut{Output}{Output}
    
\Input{Observation function $g$}
Recover eigenvalues $\sigma_p(\calk)$ and the corresponding eigenvectors $\phi_\lambda$, $\lambda \in \sigma_p(\calk)$ \;
Sample $N$ observations $\{ g(y_i)\}_{i=1}^N, \{ \phi_\lambda(y_i) \}_{i=1}^N$ for $\lambda \in \sigma_p(\calk)$ \;
Solve least square approximation problem \eqref{eqn:leastsquare}\;
\Output{Expanding coefficients $c^g_\lambda$}
\caption{Expansion Coefficients Estimation}
\label{alg:2}
\end{algorithm}

\subsubsection{Convergence Theorem}

It is well known that $\nu_{f,g,N}(\ta)$ converge weakly to $\nu_{f,g}(\ta)$. Moreover, the distance between them is in the order of $\log (N)/N$, as demonstrated in~\cite{wanner2022robust} through a straightforward estimation. Moreover, 
(Lebegue-)almost everywhere, pointwise convergence of the density can also be established under proper conditions such as the Dini criterion, see, e.g., Theorem 3.1 in~\cite{wanner2022robust}. In conclusion, we have the following theorem.
\begin{thm}
\label{thm:DS_conv}
Under the assumption that sample average recovers ${\hat \nu}_{f,g}(n)$
for each $n=1,2,\ldots, N$ with sufficient data, $\nu_{f,g,N}(\ta)$, the output of Algorithm~\ref{alg:cap},  recovers the spectral measure of $\calk$ as $N\ra \infty$.
\end{thm}

\subsection{Asymptotic Optimality of the Data-Driven Control Policy}
\label{sec:asm_opt}

With Theorem~\ref{thm:DS_conv} (weakly or pointwisely) ensuring that the data-driven process recovers the key characteristics of the original dynamical system for any fixed control policy, in this section, we demonstrate that this convergence can be carried over to the control problems. Intuitively speaking, we can solve the truncated optimal control system, and the resulting control will produce a performance that is asymptotically optimal.  More specifically, we have the following theorem.
\begin{thm}
\label{thm:convergence_of_the_DS}
Under the same assumptions for Theorem~\ref{thm:DS_conv}, the optimal control policy for the data-driven controlled problem is asymptotically optimal for the original control problem. Suppose that $J^*$ will be the optimal objective for~\eqref{eqn:MFC}, and $J^*_N$ is the optimal control objective obtained for the $N$-th data-driven system, then we have, for any $\epsilon>0$, 
$\lim\inf_{N\ra \infty} J^*_N +2\epsilon \ge J^*\ge \lim\sup_{N\ra \infty} J^*_N -2\epsilon $, hence, $\lim_{N\ra \infty} J^*_N =J^*$
\end{thm}
\begin{proof}
First, let us show that $\lim\inf_{N\ra \infty} J^*_N \ge J^*$. For each $N$, there will be a control $u_N$ such that $J_N \le J^*_N+\epsilon$. Applying $u_N$ to the limiting system will definitely dominate $J^*$. Meanwhile, for $N$ large enough $|J(u_N)-J_N(u_N)|<\epsilon$ by continuity of the cost functions and Theorem~\ref{thm:DS_conv}.  Hence, $\lim\inf_{N\ra \infty} J^*_N \ge \lim J_N(u_N) -\epsilon\ge \lim J(u_N) -2\epsilon\ge J^* -2\epsilon$. Meanwhile, for any given $\epsilon>0$, there exists a $u$ such that $J(u) \le J^*+\epsilon$ for the limiting system. Apply $u$ to any system $N$, then from Theorem~\ref{thm:DS_conv}, we know that for $N$ large enough $|J_N(u) -J(u)|<\epsilon$. Meanwhile, $J^*_N\le J(u)$ by definition. Put them together, we have, $\lim\sup_{N\ra \infty} J^*_N\le J^*+2\epsilon$. 
\end{proof}

\section{Model Predictive Control for Mean Field Control Problems}
\label{sec:MPC}

Our analysis of the stochastic Koopman operator presents methods for effective predictions for nonlinear stochastic dynamical systems. These predictive methods, especially the ones for controlled systems, enable the application of sophisticated control techniques. Here, we give details on one such approach of \emph{model predictive control}. The essence of model predictive control is to formulate and solve an optimization problem over a predictive horizon, where the objective captures the control function, and constraints reflect the dynamics of the state and control variables. The Koopman operator approach has the benefit of being essentially a linear system, hence leading to a computationally more tractable treatment. 

First, we establish an abstract formulation of a model predictive control to the Mean Field Control problem via the Koopman operator approach, assuming that the decomposition~\eqref{eqn:spectrum_exp} is available. Next, we will set up a fully data-driven approach to the mean-field control problem with the help of 
linear approximation methods such as those can be found in e.g.~\cite{korda2018linear,EDMD2015}.

In our mean-field control setting, at each time period $k$, given the state of the system being $x_k, \mu_k$, and the prediction horizon $T>0$, the following optimization problem for variables $y_\ell=(x_\ell, \mu_\ell, u_\ell, \rho_\ell)$, $\ell = k,\dots, k+T$, is considered:
\begin{align*}
    &\min_{\{y_\ell\}_{k}^{k+T}} \quad \ex\left[ \sum_{t=0}^{T-1} [F(x_{k+t}, \mu_{k+t}) + C(u_{k+t}, \rho_{k+t}) ] \right] \\
    & \hspace{4em} +\ex[G(x_{k+T},\mu_{k+T})]+\ex[H(u_{k+T}, \rho_{k+T})]\\
    \text{s.t.} & \left\{\begin{array}{cc}
        x_{\ell+1} = & b (x_\ell, \mu_\ell, u_\ell, \rho_\ell) + \sigma (x_\ell, \mu_\ell, u_\ell, \rho_\ell )w_\ell,\\ 
        \mu_{\ell+1} = & b (x_\ell, \mu_\ell), \\ 
        u_{\ell+1} =&\pi(x_{\ell+1}) +v_{\ell+1},\\ 
        \rho_{\ell+1} = &\pi(x_{\ell+1}) \end{array}\right.\\ 
    &  y_\ell = (x_\ell,\mu_\ell, u_\ell,\rho_\ell), \quad \hbox{O.C.}
\end{align*}
Here, O.C. denotes \emph{other constraints} that refer to additional constraints that accommodate application-specific requirements.  From the decomposition discussed in Sec.~\ref{sec:analysis}, we have the following representation of the dynamics from the stochastic Koopman operators method:
\begin{align*}
&\ex[F(x_{k+t}, \mu_{k+t})]\\=&\sum_{\lambda \in \sigma_p(\calk)} c^F_\g \g^t \phi_\g(y_k) + \int_{-\pi}^{\pi} e^{it\ta} \phi_{F}(\ta)(y_k) d\ta,
\end{align*}
\begin{align*}
&\ex[G(x_{k+T}, \mu_{k+T})]\\=&\sum_{\lambda \in \sigma_p(\calk)} c^G_\g \g^T \phi_\g(y_k) + \int_{-\pi}^{\pi} e^{iT\ta} \phi_{G}(\ta)(x_k) d\ta,
\end{align*}
\begin{align*}
&\ex[C(u_{k+t}, \rho_{k+t})]\\=&\sum_{\lambda \in \sigma_p(\calk)} c^C_\g \g^t \phi_\g(y_k) + \int_{-\pi}^{\pi} e^{it\ta} \phi_{C}(\ta)(y_k) d\ta,
\end{align*}
\begin{align*}
&\ex[H(u_{k+T}, \rho_{k+T})]\\=&\sum_{\lambda \in \sigma_p(\calk)} c^H_\g \g^T \phi_\g(y_k) + \int_{-\pi}^{\pi} e^{iT\ta} \phi_{H}(\ta)(y_k) d\ta.
\end{align*}
We further denote $z^\lambda_{k+t}=\g^t \phi_\g(y_k)$, $\lambda \in \sigma_p(\calk)$. Then, the optimization problem can be reformulated as
\begin{equation} \label{eqn:mpc1}
\begin{split}
    \min_{\{z^\lambda_{k+t}\}_{t=0}^T} \quad & \sum_{t=0}^{T-1} \left[ \sum_{\lambda \in \sigma_p(\calk)} (c^F_\g+c^C_\g) z^\lambda_{k+t}\right. \\ 
    &+ \left. \int_{-\pi}^{\pi} e^{it\ta} [\phi_{F}(\ta)+\phi_{C}(\ta)](y_k) d\ta \right]\\
    &+\sum_{\lambda \in \sigma_p(\calk)} (c^G_\g+c^H_\g) z^\lambda_{k+T} \\
    &+ \int_{-\pi}^{\pi} e^{iT\ta} [\phi_{G}(\ta)+\phi_H(\ta)](y_k) d\ta \\
    \text{s.t.} \quad & z^\lambda_{k+t+1}= \g z^\lambda_{k+t}, \ \ \lambda \in \sigma_p(\calk), \ t = 0,\dots, T-1, \\ 
    & \hbox{O.C.}
\end{split}
\end{equation}
Moreover, the integrals can be approximated by finite summations. More specifically, let $N$ be a large number. We have, for example for $\phi^F$,
\begin{align*}
&\int_{-\pi}^{\pi} e^{it\ta} [\phi_{F}(\ta) d]\ta \\ 
\approx \ & \frac{2\pi}{N}\sum_{n=1}^Ne^{i t \left(-\pi + n \frac{2\pi}{N}\right)} \left[\phi_{F}\left(-\pi + n \frac{2\pi}{N}\right) \right].
\end{align*}
The same applies to other functions including $\phi_C, \phi_G, \phi_H$. For, $n=1,\dots, N$, we denote 
\begin{align*}
w^{F,(n)}_{k+t}=&e^{i t \left(-\pi + n \frac{2\pi}{N}\right)} \left[\phi_{F}\left(-\pi + n \frac{2\pi}{N}\right) 
\right], \\
w^{C,(n)}_{k+t}=&e^{i t \left(-\pi + n \frac{2\pi}{N}\right)} \left[\phi_{C}\left(-\pi + n \frac{2\pi}{N}\right) 
\right],\\
w^{G,(n)}_{k+t}=&e^{i t \left(-\pi + n \frac{2\pi}{N}\right)} \left[\phi_{G}\left(-\pi + n \frac{2\pi}{N}\right) 
\right],\\
w^{H,(n)}_{k+t}=&e^{i t \left(-\pi + n \frac{2\pi}{N}\right)} \left[\phi_{H}\left(-\pi + n \frac{2\pi}{N}\right) 
\right].
\end{align*}
Since $\phi_F, \phi_C, \phi_G, \phi_H$ are the continuous eigen measures of the Koopman operator $\calk$ with respect to the corresponding functions, the dynamics of the new variables can be written by
\begin{equation}
\label{eqn:w_dynamics}
\begin{split}
    w^{F,(n)}_{k+t+1} = e^{it \frac{2\pi}{N}} w^{F,(n)}_{k+t}, \quad
    w^{C,(n)}_{k+t+1} = e^{it \frac{2\pi}{N}} w^{C,(n)}_{k+t}, \\
    w^{G,(n)}_{k+t+1} = e^{it \frac{2\pi}{N}} w^{G,(n)}_{k+t}, \quad
    w^{H,(n)}_{k+t+1} = e^{it \frac{2\pi}{N}} w^{H,(n)}_{k+t},
\end{split}
\end{equation}
for $ n = 1, \dots, N, t = 0, \dots, T-1$.
Thus, the optimization problem for model predictive control can be further as,
\begin{equation} 
\label{eqn:mpc}
\begin{split}
    \min_{z,w} \ & \sum_{t=0}^{T-1} \left[\sum_{\lambda \in \sigma_p(\calk)} (c^F_\g+c^C_\g) z^\lambda_{k+t}+ \frac{2\pi}{N} \sum_{n=1}^N [w^{F,(n)}_{k+t} + w^{C,(n)}_{k+t}]\right] \\
    & +\sum_{\lambda \in \sigma_p(\calk)} (c^G_\g+C^H_\g) z^\lambda_{k+T} + \frac{2\pi}{N}\sum_{n=1}^N [w^{G,(n)}_{k+T}+w^{H,(n)}_{k+T}] \\
    \text{s.t.} \quad & z^\lambda_{k+t+1}= \g z^\lambda_{k+t}, \quad \lambda \in \sigma_p(\calk), \ t = 0,\dots, T-1, \\
    & \text{dynamics \eqref{eqn:w_dynamics}}, \quad n = 1, \dots, N, \ t = 0, \dots, T-1, \\
    & \hbox{O.C.}
\end{split}
\end{equation}

We summarize our MPC framework based on the Koopman operator in Alg.~\ref{alg:3}.

\begin{algorithm}
Sample observation trajectory points of length $N$ using observation functions $F,G,C,H$ (add other functions if necessary) \;
Recover eigenvalues $\sigma_p(\calk)$ and the corresponding eigenfunctions $\phi_\lambda$, $\lambda \in \sigma_p(\calk)$ \;
Recover expanding coefficients $c^F_\lambda, c^G_\lambda, c^C_\lambda, c^H_\lambda$ using Alg.~\ref{alg:2} \;
Recover spectral measure $\nu_F, \nu_G, \nu_C, \nu_H$ for functions $F, G, C, H$\;
Approximate the continuous measure for the function $F$ by $\phi_F \approx \nu_F - \sum_{\lambda \in \sigma_p(\calk)} c^F_\lambda \phi_\lambda$ (same for functions $G,C,H$) \;
Compute $z_0$ from $x_0, \mu_0$ \;
\For{$k = 0, 1, 2, \dots$}{
    Formulate MPC problem \eqref{eqn:mpc} with initial condition $z_k$ \;
    $\{z^*_{\ell}, w^*_{\ell}\}_{\ell=1}^T \gets$ solve \eqref{eqn:mpc} \;
    update current state $z_{k} \gets z^*_{1}$ \;
}
\caption{Model Predictive Control Framework based on Koopman Operator}
\label{alg:3}
\end{algorithm}

In the problem \eqref{eqn:mpc}, the evolution of the system dynamics is already incorporated into the spectral decomposition of the Koopman operator. Therefore, we have a new diagonal and linear system represented by variables $z^\lambda$ to substitute the original dynamics \eqref{eqn:MFC_dynamics_2}. Besides, the evolution of the cost functions in the objective is measured in the Koopman eigenspace. So, we can use Koopman eigenfunctions and the corresponding expansion coefficients to reduce the nonlinearities, resulting in a new problem that is much easier to solve than the original mean field control problems.

\section{Conclusions}
\label{sec:conclusions}

In this paper, we proposed a stochastic Koopman operator approach to mean field control problems for systems with a large number of agents. For the analysis stochastic Koopman operator, we have adopted a spectral approach, which has seen success in both dynamical systems \cite{colbrook2024rigorous} and control \cite{Mezic2005}. This approach leads to a guarantee of convergence of the data-driven estimators to the stochastic Koopman operator in a probabilistic sense. One of the main contributions of the paper is to show that this convergence in dynamical systems will lead to the convergence of the controlled systems, thus implying that the controls derived from the estimated Koopman operator are asymptotically optimal. Subsequently, we propose a model predictive control framework based on the stochastic Koopman operator for a large family of mean field control problems. 
The mean field control problem under consideration takes a special form that was studied in~\cite{ELLIOTT20133222}, and we see the potential of the analysis in the paper being extended to mean field control in general forms, such as those studied in~\cite{MFGvMFC}. This will be an important topic for future research, along with the applications of this analysis in applications such as robotics, transportation, and biology.

\bibliographystyle{abbrv}

\end{document}